\documentclass[11pt]{amsart}
\usepackage{amssymb, bm, amsmath, latexsym, mathabx, dsfont,tikz, float, mathrsfs}
\usepackage{makecell, multirow, longtable}
\usepackage{enumitem}

\renewcommand{\baselinestretch}{\baselinestretch}
\renewcommand{\baselinestretch}{1.1}
\numberwithin{equation}{section}

\newtheorem{thm}{Theorem}[section]
\newtheorem{lem}[thm]{Lemma}
\newtheorem{cor}[thm]{Corollary}
\newtheorem{prop}[thm]{Proposition}

\theoremstyle{definition}

\newtheorem{rmk}[thm]{Remark}

\theoremstyle{remark}

\numberwithin{equation}{section}
\newcommand{\ra}{{\,\rightarrow\,}}
\newcommand{\nra}{{\,\nrightarrow\,}}
\newcommand{\gen}{\text{gen}}

\newcommand{\ord}{\textnormal{ord}}

\newcommand{\z}{{\mathbb Z}}
\newcommand{\q}{{\mathbb Q}}

\newcommand{\n}{{\mathbb N}}

\newcommand{\sqf}{{\mathrm{sf}}}

\newcommand{\mca}{\mathcal{A}}
\newcommand{\mcc}{\mathcal{C}}
\newcommand{\mcd}{\mathcal{D}}

\newcommand{\0}{\bm 0}

\newcommand{\legendre}[2]{ \left(\frac{#1}{#2}\right)}

\newcommand{\Mod}[1]{\ (\mathrm{mod}\ #1)}

\title{Composition laws of binary quadratic forms and isolations of quadratic forms}

\author[Jangwon Ju et al.] {Jangwon Ju, Daejun Kim, Kyoungmin Kim, Mingyu Kim, and Byeong-Kweon Oh}

\address{Department of Mathematics Education, Korea National University of Education, Cheongju 28173, Korea}
\email{jangwonju@knue.ac.kr}
\thanks{This work of the first author was supported by the National Research Foundation of Korea(NRF) grant funded by the Korea government(MSIT) (NRF-2022R1A2C1092314).}

\address{Department of Mathematics Education, Korea University, Seoul 02841, Republic of Korea}
\email{daejunkim@korea.ac.kr}
\thanks{This work of the second author was supported by the National Research Foundation of Korea(NRF) grant funded by the Korea government(MSIT) (RS-2024-00342122 and RS-2024-00455692).}

\address{Department of Mathematics, Hannam University, Daejeon 34430, Korea}
\email{kiny30@hnu.kr}

\address{Department of Mathematics Education, Pusan National University, Busan 46241, Korea}
\email{mingyukim@pusan.ac.kr}
\thanks{This work of the third and the fourth author was supported by the National Research Foundation of Korea(NRF) grant funded by the Korea government(MSIT) (NRF-2021R1C1C2010133).}

\address{Department of Mathematical Sciences and Research Institute of Mathematics, Seoul National University, Seoul 08826, Korea}
\email{bkoh@snu.ac.kr}
\thanks{This work of the fifth author was supported by the National Research Foundation of Korea(NRF) grant funded by the Korea government(MSIT) (NRF-2020R1A5A1016126 and RS-2024-00342122).}



\begin{document}

\begin{abstract} A positive definite and integral quadratic form $f$ is called irrecoverable if there is a quadratic form $F$ such that it  represents all proper subforms of $f$, whereas it does not represent $f$ itself. In this case, $F$ is called an isolation of $f$.  In this article, we prove that there does not exist a binary isolation of any unary quadratic form. We also prove that there does not exist a ternary isolation of any binary quadratic form. Furthermore, if the form class group of a primitive binary quadratic form has no element of order $4$, then the discriminant of any quaternary isolation of it, if exists, is a square of an integer. The composition laws of primitive binary quadratic forms play an essential role in the proofs of the results.
\end{abstract}
\maketitle

\section{Introduction} A quadratic form $f$ of rank $n$, for a positive integer $n$, is a homogeneous quadratic polynomial
$$
f(x_1,x_2,\dots,x_n)=\sum_{i,j=1}^n f_{ij} x_ix_j, \quad  (f_{ij}=f_{ji} \in \q),
$$ 
where the symmetric matrix $M_f=(f_{ij})$ which is called the Gram matrix of $f$ is non-degenerate. The {\it scale ideal} $\mathfrak{s}f$ and the {\it norm ideal} $\mathfrak{n}f$ of $f$ are defined to be the $\z$-module generated by $\{f_{ij}\mid 1\le i,j\le n\}$ and $\{f_{ii}\mid 1\le i\le n\}$, respectively. We say that $f$ is {\it integral} if $\mathfrak{s}f\subseteq\z$.  The determinant of $M_f$ is called the {\it discriminant} of $f$, which is denoted by $df$. In particular, if $f(x,y)=ax^2+bxy+cy^2$ is a binary quadratic form, then we additionally use the classical definition for its discriminant $D_f:=-4df=b^2-4ac$. If $M_f$ is a diagonal matrix, then we simply write $f=\langle f_{11},f_{22},\dots,f_{nn}\rangle$. Throughout this article, we always assume that the Gram matrix $M_f$ is positive definite.   

Let $f$ and $g$ be quadratic forms with ranks $n$ and $m$, respectively. We say $g$ is {\it represented} by $f$, and write $f\ra g$, if there is an $n\times m$ integral matrix  $T=(t_{ij}) \in M_{n\times m}(\z)$ such that 
\begin{equation}\label{eqn:def-rep}
g(y_1,y_2,\dots,y_m)=f(t_{11}y_1+t_{12}y_2+\cdots+t_{1m}y_m,\dots,t_{n1}y_1+\dots+t_{nm}y_m).    
\end{equation}
If we denote the Gram matrices of $f$ and $g$ by $M_f$ and $M_g$, respectively, then this implies that 
\[
    T^tM_fT=M_g.
\]
If the above matrix $T$ can be taken from $\mathrm{GL}_n(\z)$, then we say that $f$ is {\it isometric} to $g$. For each prime $p$, we say that $g$ is represented by $f$ over $\z_p$, if there is a $T \in M_{n\times m}(\z_p)$ satisfying \eqref{eqn:def-rep}, where $\z_p$ is the ring of $p$-adic integers. Similarly, $f$ is isometric to $g$ over $\z_p$ if $T\in\mathrm{GL}_n(\z_p)$. We also say $g$ is {\it locally represented} by $f$ if $g$ is represented by $f$ over $\z_p$ for any prime $p$. The {\it genus} of $f$ is the set of all quadratic forms that are locally isometric to $f$. The {\it class number} of $f$ is the number of quadratic forms in the genus of $f$ that are not isometric over $\z$. It is well known that the class number of $f$ is always finite and is greater than $1$ if the rank of $f$ is greater than $11$ (see \cite[Theorem 1]{Wat} or \cite{LK}).

One of the main problems on the representations of quadratic forms is to find all quadratic forms that are represented by a given quadratic form. If a quadratic form $f$ has class number $1$, then any quadratic form is represented by $f$ if and only if it is locally represented by $f$. Furthermore, there is an effective method to find all quadratic forms that are locally represented by a quadratic form (for this, see \cite{om1}). Indeed, most quadratic forms of rank greater than $1$ have class numbers greater than $1$. If a quadratic form $f$ is of rank greater than or equal to $4$, then there is an effective method to find all unary quadratic forms that are represented by $f$, if the discriminant of $f$ is not too big.  To the authors' knowledge, there is no effective method to find all quadratic forms of rank greater than $1$ that are represented by a quadratic form of class number greater than $1$, except for some very special cases (for example, see \cite{kko}).

Let $f$ be an integral quadratic form. We say that $f$ is {\it irrecoverable} if there is an ``integral" quadratic form $F$ such that it represents all proper subforms of $f$, whereas it does not represent $f$ itself. If such an $F$ exists, then it is called an {\it isolation} of $f$. If there does not exist an isolation, then $f$ is called {\it recoverable}. The notion of (ir)recoverability of a quadratic form was first introduced in \cite{o}. In \cite{klo}, they study the close relationship between this concept and the notion of a minimal universality criterion set for the set of all proper subforms of $f$. In the same article, various properties of (ir)recoverable binary quadratic forms are also investigated. One may easily show that the cubic quadratic form $I_n=\langle1,\ldots, 1\rangle$ of rank $n$ is irrecoverable for any positive integer $n$. For isolations with minimal rank of $I_n$, see \cite{jkkko} and \cite{o}. Recently, it was proved in \cite{co} that any indecomposable quadratic form is irrecoverable.           

For an irrecoverable quadratic form $f$, the minimal rank of isolations of $f$ is denoted by $\text{Iso}(f)$.  One may easily show that the ternary quadratic form $g=2x^2+2y^2+5z^2$ represents all squares of integers greater than $1$ (see \cite{jko}). Since $1$ is not represented by $g$, $g$ is an isolation of the unary quadratic form $I_1=x^2$. Therefore, the ternary quadratic form $2mx^2+2my^2+5mz^2$ is an isolation of $mx^2$, which implies that $\text{Iso}(mx^2) \le 3$ for any positive integer $m$.

In this article, we show that there does not exist a binary isolation of any unary quadratic form (see Theorem \ref{thm:no-binary-iso}).
Moreover, for a positive integer $n$ which is not represented by $f$ with $\mathfrak{n}f=\z$, we show that even though $f$ is not an isolation of $n$, there are infinitely many primes $p$ such that $np^2\ra f$ and classified all such primes $p$ not dividing $D_f$ (see Theorem \ref{thm:np^2-classify}). 

For any irrecoverable binary quadratic form $f$, we show that there is no ternary isolation of $f$, and hence $\text{Iso}(f) \ge 4$ (see Corollary \ref{ternary}). We also show that if a quaternary form $F$ is an isolation of an irrecoverable binary quadratic form $f$ with $\mathfrak{n}f=s\z$, then the squarefree part of $dF$ divides $D_f/s^2$ (see Corollary \ref{cor:quat}). In particular, if $f=kg$, where $k$ is a positive integer and $\mathfrak{n}g=\z$, and the form class group $\mathfrak S_{D_g}$ contains no element of order $4$, then the discriminant of any quaternary isolation of $f$ is a square of an integer (see Corollary \ref{4-square}).

To prove our results, the composition laws of primitive binary quadratic forms play an essential role. For the definition and some properties of composition laws of primitive binary quadratic forms, one may consult  \cite{ca}, \cite{cox}, or \cite{hua}.     

The subsequent discussion will be conducted in the language of quadratic spaces and lattices. By a $\z$-lattice, we mean a free $\z$-module $L=\z e_1 + \cdots +\z e_n$ together with a non-degenerate bilinear form $B$. The associated quadratic map $Q$ on $L$ is defined by $Q(x)=B(x,x)$. For a chosen basis $\{e_i\}$ of $L$, the quadratic form $f_L$ associate to $L$ is defined by 
\[
f_L(x_1,\ldots,x_n)=\sum_{1\le i,j\le n} B(e_i,e_j)x_ix_j.
\]
The Gram matrix, discriminant, scale ideal, norm ideal, and other invariants of $L$ are defined in the same way as for the corresponding quadratic form $f_L$.

We always assume that $\mathfrak{s}L\subseteq \z$. However, in the binary case, we consider binary $\z$-lattices $\ell$ with $\mathfrak{n}\ell\subseteq\z$ and set $D_\ell=-4d\ell$, following the classical convention for the binary quadratic forms. Note that $D_\ell\equiv 0,1\Mod{4}$.

The reader is referred to \cite{ki} and \cite{om} for any unexplained notation and terminology.    

\section{Binary isolations of unary lattices}

In this section, we show that there is no binary isolation of unary lattices, and discuss further related properties.

\subsection{Preliminaries}
Let $D$ be a negative integer such that $D \equiv 0,1 \Mod 4$. Let $\mathfrak S_D$ be the set of proper equivalence classes of positive definite binary quadratic forms $f$ with $\mathfrak{n}f=\z$ and  discriminant $D_f=D$. Then it is well known that $\mathfrak{S}_D$ forms a finite abelian group. The identity class in $\mathfrak{S}_D$ is denoted by $\mathcal{E}$. A class $\mathcal{C} \in \mathfrak{S}_D$ is called {\it an ambiguous class} if $\mathcal C^2=\mathcal{E}$. We define 
$$
\gen(\mathcal E)=\{\mathcal C^2 : \mathcal C \in \mathfrak S_D\},
$$
which is called {\it the principal genus}.  Furthermore, for any $\mathcal A \in \mathfrak S_D$, we define
$$   
\gen(\mathcal A)=\{ \mathcal A\cdot\mathcal C : \mathcal C \in \gen(\mathcal E)\}.
$$

For a class $\mcc$, let $f$ be a quadratic form in $\mcc$. For an integer $n$, we write $n\ra\mcc$ if $n$ is represented by $f$. Moreover, we write $n\ra\mcc_p$ if $n$ is represented by $f$ over $\z_p$. If $n$ is locally represented by $f$, then we write $n\ra \gen(\mcc)$. Since any quadratic form in $\mcc$ represents the same set of integers, the above terminologies are well-defined. We denote the set of all integers that are represented by $\mcc$ as $Q(\mcc)$.

Let $f$ be a binary quadratic form and let $n$ be a positive integer. We say that an integer solution to $f(x,y)=n$ is {\it primitive} if $(x,y)=1$. The following lemma counts the total number of primitive solutions to binary quadratic forms in a form class group.

\begin{lem}\label{lem:total-repnum}
    Let $n$ be a positive integer and let $D$ be a negative integer with  $D\equiv 0,1\Mod{4}$. Assume that $(n,D)=1$. Then the total number of primitive representations of $n$ by class representatives $f_1,\ldots,f_{h(D)}$ of $\mathfrak{S}_D$ is given by
    \begin{equation}\label{eqn:total-repnum}
        \psi(n)=w\sum_{t\mid n}\legendre{D}{t},
    \end{equation}
    where $w=6$ if $D=-3$, $w=4$ if $D=-4$, and $w=2$ otherwise.
\end{lem}
\begin{proof}
    See Theorem 4.1 in \cite[Chapter 12]{hua}.
\end{proof}

\begin{lem}
Let $\ell$ be a binary $\z$-lattice with $\mathfrak{n}(\ell)=\z$, and let $a$ and $b$ be positive integers such that $(a,b)=1$. Let $\mathcal{A}_{a,b}=\{an+b \mid n\in\z^+\}$ be the set of an arithmetic progression. Assume that $\mathcal{A}_{a,b}\cap Q(\gen(\ell))\neq\varnothing$. Then the set $\mathcal{A}_{a,b}\cap Q(\gen(\ell))$ contains infinitely many primes.
\end{lem}
\begin{proof}
     This is a direct consequence of \cite{m}.
\end{proof}

\subsection{Binary isolations of unary lattices}

Let us consider a binary $\z$-lattice $\ell$ with $\mathfrak{n}\ell=\z$. By the form class corresponding to $\ell$, we mean the form class $\mathcal{C} \in \mathfrak{S}_{D_\ell}$ such that $f_\ell\in  \mathcal{C}$. Note that $f_\ell$ depends on the choice of the basis for $\ell$. For example, if we consider two bases $\{e_1,e_2\}$ and $\{e_1,-e_2\}$ of $\ell$, both $\mathcal{C}$ and $\mathcal{C}^{-1}$ may correspond to the same lattice $\ell$. Thus, the choice of $\mathcal{C}$ corresponding to $\ell$ depends on the choice of basis. When there is no ambiguity, we simply denote the form class corresponding to $\ell$ by $\mathcal{C}_\ell$.

\begin{prop}\label{prop:iso-unary}
Let $n$ be a positive integer and let $\ell$ be a binary $\z$-lattice with $\mathfrak{n}\ell=\z$. Let $p$ be a prime such that either
	\begin{enumerate}[label={\rm{(\arabic*)}}]
		\item $\left(\frac{D_\ell}{p}\right) =-1$ or
		\item $\left(\frac{D_\ell}{p}\right) = 1$ and $p\in Q(\mcd)$ for some ambiguous class $\mcd\in \mathfrak{S}_{D_\ell}$.
	\end{enumerate}
	If $np^2$ is represented by $\ell$, then $n$ is represented by $\ell$.
\end{prop}
\begin{proof}
We first assume that $\left(\frac{D_\ell}{p}\right)=-1$ and let $Q(x)=np^2$ for some $x\in \ell$. From the assumption,  we have $\ell_p\cong \langle 1,-\Delta_p\rangle$ which is anisotropic over $\z_p$. Here $\Delta_p$ is a quadratic non-residue in $\z_p^\times$. Hence we have $x=py$ for some $y\in\ell$ and $Q(y)=n$. Thus $n$ is represented by $\ell$.
	
Now, assume that $p\in Q(\mcd)$ for some ambiguous class $\mcd\in \mathfrak{S}_D$ and let $\mathcal{C}=\mathcal{C}_\ell$. Since $np\cdot p \in Q(\mcc)$ and $p\in Q(\mcd)$, it follows from \cite[Lemma 2.2]{ef} that either $np\in Q(\mcc\mcd)$ or $np \in Q(\mcc{D^{-1}})$. Applying it once again, one may conclude that $n$ is represented by either
	$$
	\mcc\mcd\cdot\mcd,\ \mcc\mcd\cdot\mcd^{-1},\ \mcc\mcd^{-1}\cdot\mcd, \text{ or } \mcc\mcd^{-1}\cdot\mcd^{-1}.
	$$
	Since $\mcd^2=\mathcal{E}_D$, all the above four classes are equal to $\mcc$. Thus $n$ is represented by $\ell$.
\end{proof}

\begin{thm}\label{thm:no-binary-iso}
    Let $n$ be a positive integer. Then there is no binary isolation of $\langle n \rangle$.
\end{thm}
\begin{proof}

Assume that $\ell$ is a binary isolation of $\langle n \rangle$. Without loss of generality, we may assume that $\mathfrak{n}\ell=\z$. Note that $\ell$ represents $np^2$ for a prime $p$ satisfying $\legendre{D_\ell}{p}=-1$. Hence by Proposition \ref{prop:iso-unary}, $\ell$ represents $n$. This is a contradiction.
\end{proof}

Let $n$ be a positive integer and $D$ be a negative integer with $D\equiv 0,1\Mod{4}$. For a binary $\z$-lattice $\ell$ with $\mathfrak{n}\ell=\z$ and $D_\ell=D$ such that $n$ is not represented by $\ell$, we characterize the primes $p$ for which $np^2$ is represented by $\ell$, even though $\ell$ is not an isolation of $n$.

\begin{lem}\label{lem:rep-by-S_D}
    Let $n\in\n$ and let $\sqf(n)$ denote the squarefree part of $n$. Let $D$ be a negative integer with $D\equiv0,1\Mod{4}$ and $(n,D)=1$. The following are equivalent:
    \begin{enumerate}[label={\rm(\arabic*)}]
        \item No binary $\z$-lattice with $\mathfrak{n}\ell=\z$ and $D_\ell=D$ represents $n$.
        \item There is a prime $p\mid\sqf(n)$ such that $\left(\frac{D}{p}\right)=-1$.
    \end{enumerate}
\end{lem}
\begin{proof}
    First, we show that (1) implies (2). If $\sqf(n)=1$, that is, $n=m^2$ for some $m\in\n$, then we have $m^2\in Q(\mathcal{E})$. Hence we may assume that $\sqf(n)\neq1$. Note that any representation of $\sqf(n)$ by a binary $\z$-lattice is primitive. Hence by Lemma \ref{lem:total-repnum} the total number of representations of $\sqf(n)$ by classes in $\mathfrak{S}_D$ is 
    \[
        \psi(\sqf(n))=w\sum_{t\mid \sqf(n)} \legendre{D}{t} = w \prod_{\substack{q\mid \sqf(n) \\ q\text{ prime}}} \left(1+\legendre{D}{q}\right),
    \]
    where we use the multiplicativity of the Kronecker symbol. Moreover, $\psi(\sqf(n))$ equals to zero, since $\sqf(n)$ should not be represented by any primitive binary $\z$-lattice of discriminant $D$. Thus, at least one prime $q$ dividing $\sqf(n)$ should satisfy $\legendre{D}{q}=-1$, which implies (2).

 To show that  (2) implies (1), write $n=\sqf(n)m^2$ with $m\in\n$. Note that the total number $N$ of representations of $n$ by classes in $\mathfrak{S}_D$ is given by
    \[
        N=\sum_{k\mid m} \psi\left(\frac{n}{k^2}\right)=\sum_{k\mid m} \psi\left(\sqf(n)k^2\right)=\sum_{k\mid m}w \prod_{\substack{q\mid \sqf(n)k^2 \\ q\text{ prime}}} \left(1+\legendre{D}{q}+\cdots+\legendre{D}{q}^{\ord_q(\sqf(n)k^2)}\right).
    \]
    Note that every product in the sum is zero since there is a prime $p$ dividing $\sqf(n)$ such that $\legendre{D}{p}=-1$ and $\ord_p(\sqf(n)k^2)$ is an odd integer for any positive integer $k$ dividing $m$. Therefore $N=0$, which implies (1).
\end{proof}

\begin{lem}\label{lem:rep-by-mcc}
    Let $D$ be a negative integer with $D\equiv 0,1\Mod{4}$, and let $\mcc\in\mathfrak{S}_D$. Let $n$ be a positive integer and let $p$ be a prime not dividing $D$ such that $np^2\ra \mcc$. We have
    \begin{enumerate}[label={\rm(\arabic*)}]
        \item $n\ra \gen(\mcc)$.
        \item If $n\nra \mcc$, then $\legendre{D}{p}=1$. 
    \end{enumerate}

\end{lem}
\begin{proof}
If $\legendre{D}{p}=-1$, then the same argument as in Proposition \ref{prop:iso-unary} implies that $n\ra \mcc$. This proves part (2) and also implies that $n\ra \gen(\mcc)$. If $\legendre{D}{p}=1$, then we have
\[
\mcc_p\cong \left(\begin{smallmatrix} 0& \frac{1}{2} \\ \frac{1}{2} & 0\end{smallmatrix}\right),
\] 
and hence $n\ra \mcc_p$. For each prime $q\neq p$, we have $n\ra\mcc_q$ since $np^2\ra \mcc$. Thus $n\ra \gen(\mcc)$. 
\end{proof}

\begin{thm}\label{thm:np^2-classify}
Let $D$ be a negative integer with $D\equiv 0,1\Mod{4}$ and let $\mcc\in\mathfrak{S}_D$. Let $n$ be a positive integer that is not represented by $\mcc$.

	\begin{enumerate}[label={\rm(\arabic*)}]
    \item If $n$ is not represented by any class in $\mathfrak{S}_D$, then neither is $np^2$ for any prime $p$. In particular, $np^2\nra\mcc$ for any prime $p$.
	\item Assume that $n$ is represented by a class in $\mathfrak{S}_D$. If $n\nra \gen(\mcc)$ and $np^2\ra\mcc$ for some prime $p$, then $p$ divides $D$.
		\item Assume that $n\ra \gen(\mcc)$ and let $p$ be a prime. Then $np^2\ra\mcc$ and $p$ does not divide $D$ if and only if $p\ra\mcd$ for some $\mcd\in\mathfrak{S}_D$ such that $n\ra \mcc\mcd^2\in\gen(\mcc)$.
       
	\end{enumerate}
\end{thm}
\begin{proof}
    Part (1) follows directly from Lemma \ref{lem:rep-by-S_D}. Also, part (2) is nothing but Lemma \ref{lem:rep-by-mcc} (1). 
    
    To show (3), we first consider a prime $p$ not dividing $D$ such that $np^2\ra\mcc$. Note that $\legendre{D}{p}=1$ by Lemma \ref{lem:rep-by-mcc} (2). Since 
    \[
        \psi(p)=w\sum_{d\mid p} \legendre{D}{d}=w\left(1+\legendre{D}{p}\right)=2w>0,
    \]
    there is a class  $\mcd\in\mathfrak{S}_D$ such that $p\in Q(\mcd)$. Applying  \cite[Lemma 2.2]{ef}, we have either $np\in Q(\mcc\mcd)$ or $np \in Q(\mcc{D^{-1}})$. Applying it once again, the integer $n$ is represented by at least one of
	$$
	\mcc\mcd\mcd,\ \mcc\mcd\mcd^{-1}(=\mcc),\ \mcc\mcd^{-1}\mcd(=\mcc), \text{ and } \mcc\mcd^{-1}\mcd^{-1}.
	$$
    Since we are assuming $n\nra \mcc$, either $n\ra\mcc\mcd^2$ or $n\ra\mcc\mcd^{-2}$.  
    Hence the ``only if" part of (3) follows from the fact that $Q(\mcd)=Q(\mcd^{-1})$.

    To show the ``if" part, let $p$ be a prime that is represented by $\mcd\in\mathfrak{S}_D$ such that $n\ra\mcc\mcd^2$. Since $p\in Q(\mcd)=Q(\mcd^{-1})$, we have $np^2\in Q(\mcc\mcd^2\mcd^{-2})=Q(\mcc)$. Now, it suffices to show that $p$ does not divide $D$. Assume to the contrary that $p$ divides $D$. Since $p\in Q(\mcd)$, we have
    \[
        \mcd \cong \begin{pmatrix} p& b \\ b & c  \end{pmatrix} \quad \text{and} \quad D=(2b)^2-4pc\equiv 0 \Mod{p}.
    \]
    Hence $2b=pk$ for some $k\in\z$, and therefore we have
    \[
        \mcd \cong \begin{cases} 
        \begin{pmatrix} p& 0 \\ 0 & -\frac{D}{4p}  \end{pmatrix} & \text{if } b\in\z, \\[15pt]
        \begin{pmatrix} p& \frac{p}{2} \\ \frac{p}{2} & \frac{p^2-D}{4p}  \end{pmatrix} & \text{if } b\in\frac{1}{2}+\z.
        \end{cases}
    \]
    In both cases, $\mcd$ is an ambiguous class. However, since $np^2\in Q(\mcc)$ and $p\in Q(\mcd)$, one may apply \cite[Lemma 2.2]{ef} as in the proof of Proposition \ref{prop:iso-unary} to show that $n\in Q(\mcc)$. This contradiction completes the proof of the theorem.
\end{proof}

%%%%%%%%%%%%%%%%%%%%%%%%%%%%%%%%%%%%%%%%%%%%%%%%%%%%%%%%%%%%%%%%%%%%%%%%%%%%%%%%%%%%%%%
%%%%%%%%%%%%%%%%%%%%%%%%%%%%%%%%%%%%%%%%%%%%%%%%%%%%%%%%%%%%%%%%%%%%%%%%%%%%%%%%%%%%%%%%
%%%%%%%%%%%%%%%%%%%%%%%%%%%%%%%%%%%%%%%%%%%%%%%%%%%%%%%%%%%%%%%%%%%%%%%%%%%%%%%%%%%%%%%%
%%%%%%%%%%%%%%%%%%%%%%%%%%%%%%%%%%%%%%%%%%%%%%%%%%%%%%%%%%%%%%%%%%%%%%%%%%%%%%%%%%%%%%%%
%%%%%%%%%%%%%%%%%%%%%%%%%%%%%%%%%%%%%%%%%%%%%%%%%%%%%%%%%%%%%%%%%%%%%%%%%%%%%%%%%%%%%%%%

\section{Isolations of irrecoverable binary lattices}\label{sec:Iso-of-binary}
Recall that any indecomposable $\z$-lattice is irrecoverable and therefore has an isolation. For a proof of this in the binary case, see \cite{klo}, and in the general case, see \cite{co}. However, little is known about the minimal rank of isolations of irrecoverable lattices. In this section, we discuss general properties of isolations of binary lattices.

\begin{lem} \label{easy} Let $a$ be a nonsquare positive integer and let $b$ be any positive integer. Then there is an arithmetic progression $\mathfrak P$ such that for almost all primes $p \in \mathfrak P$, 
\begin{equation} \label{exist}
\left(\frac ap\right)=-1 \quad \text{and} \quad \left(\frac{-b}p\right)=1.
\end{equation}
\end{lem}

\begin{proof} Assume that $a=a_1^2\prod_{i=1}^rp_i$ and $b=b_1^2\prod_{j=1}^sq_j$, where $p_i$ and $q_j$ are primes and $r\ge 1$. Suppose that there is an $i_0$ with $1\le i_0\le r$ such that $p_{i_0} \nmid \prod_{j=1}^sq_j$. Let $p_{i_0}^*=p_{i_0}$, when $p_{i_0}$ is odd, and $p_{i_0}^*=8$ otherwise.
Let $\delta$ be a positive integer less than $p_{i_0}^*$ such that $\left(\frac{\delta}{p_{i_0}}\right)=-1$ if $p_{i_0}$ is odd, and $\delta=5$ otherwise.
Consider the arithmetic progression $\mathfrak P$ consisting of integers $d$  satisfying
$$
d \equiv 1 \ \left(\text{mod} \  4\prod_{j=1}^s q_j\cdot p_1\cdots p_{i_0-1}p_{i_0+1}\cdots p_r\right)\quad \text{and} \quad d \equiv \delta \Mod {p_{i_0}^*}. 
$$
Then one may easily check that any prime $p \in \mathfrak P$ not dividing $a_1b_1$ satisfies Equation \eqref{exist}. 
Since the proof of the case when there is a $q_{j_0}$ such that $q_{j_0} \nmid \prod_{i=1}^r p_i$ is quite similar to this,  the proof is left to the reader. Finally, assume that $\prod p_i=\prod q_j$. In this case, one may choose a suitable arithemetic progression  that is contained in the set of integers congruent to $3$ modulo $4$. This completes the proof.
 \end{proof}

%%%%%%%%%%%%%%%%%%%%%%%%%%%%%%%%%%%%%%%%%%%%%%%%%%%%%%%%%%%%%%%%%%%%%%%%%%%%%%%%%%%%%%%%
%%%%%%%%%%%%%%%%%%%%%%%%%%%%%%%%%%%%%%%%%%%%%%%%%%%%%%%%%%%%%%%%%%%%%%%%%%%%%%%%%%%%%%%%
%%%%%%%%%%%%%%%%%%%%%%%%%%%%%%%%%%%%%%%%%%%%%%%%%%%%%%%%%%%%%%%%%%%%%%%%%%%%%%%%%%%%%%%%
%%%%%%%%%%%%%%%%%%%%%%%%%%%%%%%%%%%%%%%%%%%%%%%%%%%%%%%%%%%%%%%%%%%%%%%%%%%%%%%%%%%%%%%%
%%%%%%%%%%%%%%%%%%%%%%%%%%%%%%%%%%%%%%%%%%%%%%%%%%%%%%%%%%%%%%%%%%%%%%%%%%%%%%%%%%%%%%%%

\begin{lem} Let $D$ be a negative integer with $D\equiv 0,1\Mod{4}$ and let $\mathfrak{S}_D$ be the form class group of discriminant $D$. Then every genus in $\mathfrak S_D$ contains an ambiguous class if and only if there is no element in $\mathfrak S_D$ of order $4$.  
\end{lem}

\begin{proof} Suppose that $\mathfrak S_D$ does not contain an element of order $4$. Then 
$$
\mathfrak S_D \simeq \displaystyle \bigoplus_{i=1}^t \z/2\z\oplus H,
$$
where $H$ is a subgroup of $\mathfrak S_D$ with odd order. One may easily check that $\gen(\mathcal E)=H$, and any genus in $\mathfrak S_D$ is of the form $(\eta_1,\eta_2,\dots,\eta_t)+H$, where $\eta_i \in \{0,1\}$ for any $i=1,2,\dots,t$. Therefore each genus contain an ambiguous class of the form $(\eta_1,\eta_2,\dots,\eta_t)$. 

Conversely, assume that each genus in $\mathfrak S_D$ contains an ambigulus class. If $\mathfrak S_D$ contains  an element of order $4$, say $\mathcal C$, then the ambiguous class $\mathcal C^2$ is contained in the principal genus. Hence the principal genus contains at least $2$ ambiguous classes, since the number of genera in $\mathfrak{S}_D$ coincides with the number of ambiguous classes (for this, see \cite[Theorem 3.15 (i)]{cox}). This is a contradiction to the fact that the number of ambiguous classes in $\mathfrak S_D$ is equal to the number of genera in $\mathfrak S_D$.   
\end{proof}

\begin{lem} \label{useful} Let $D$ be a negative integer such that $D \equiv 0,1 \Mod 4$. Let $p$ be a prime such that $\left(\frac Dp\right)=1$. For any binary $\z$-lattice $\ell$ with $\mathfrak{n}\ell=\z$ and $D_\ell=D$, the following hold:

\begin{enumerate}[label={\rm{(\arabic*)}}]
	\item there are exactly two sublattices $\ell(p,1)$ and $\ell(p,2)$ of $\ell$ with index $p$ such that $\mathfrak{n}\ell(p,i) \subseteq p\z$ for  $i=1,2$;
	
	\item if $\ell \in \mathcal C$ and $p \in Q(\mathcal D)$ for some $\mcd\in\mathfrak{S}_D$, then we have
	\[
		\ell(p,i)^{\frac1p} \in \mathcal C \cdot\mcd \text{ and } \ell(p,j)^{\frac1p} \in \mathcal C\cdot \mcd^{-1}
	\]
	for some $i,j$ with $\{i,j\}=\{1,2\}$.
\end{enumerate}
\end{lem}

\begin{proof} Since $\ell_p \simeq \begin{pmatrix} 0&\frac{1}{2}\\\frac{1}{2}&0\end{pmatrix}$, there is a basis $\{\bm{x},\bm{y}\}$ for $\ell$ such that 
$$
\ell=\z\bm{x}+\z \bm{y}=\begin{pmatrix} pa&b\\b&pc\end{pmatrix} \in \mathcal C.
$$
Note that all sublattices of $\ell$ with index $p$ are 
$$
\z p\bm{x}+\z \bm{y}, \  \z (\bm{x}+u\bm{y})+\z p\bm{y},
$$
where $0\le u\le p-1$. Among these sublattices, only $\ell(p,1)=\z p\bm{x}+\z \bm{y}$ and  $\ell(p,2)=\z \bm{x}+\z p\bm{y}$ have norms which are contained in $p\z$.

Now, we show the second statement. Since $p \in Q(\mathcal D)$, we may assume that 
$$
\begin{pmatrix} p&b_1\\b_1&c_1\end{pmatrix} \in \mathcal D \quad \text{and} \quad   \begin{pmatrix} p&-b_1\\-b_1&c_1\end{pmatrix} \in \mathcal D^{-1}.
$$
Since $D_\ell=4b^2-4p^2ac=4b_1^2-4pc_1$, we may further assume without loss of generality that $2b-2b_1 \equiv 0 \Mod p$. Hence we have
$$
\begin{pmatrix} p&b_1\\b_1&c_1\end{pmatrix} \simeq  \begin{pmatrix} 1&0\\ \frac{b-b_1}p&1\end{pmatrix}\begin{pmatrix} p&b_1\\b_1&c_1\end{pmatrix}\begin{pmatrix} 1& \frac{b-b_1}p\\ 0&1\end{pmatrix}=\begin{pmatrix}p&b\\b&pac\end{pmatrix} \in \mathcal D.
$$
Therefore we have 
$$
\ell(p,1)^{\frac1p}=\begin{pmatrix} p^2a&b\\b&c\end{pmatrix} = \begin{pmatrix} pa&b\\b&pc\end{pmatrix} \cdot \begin{pmatrix}p&b\\b&pac\end{pmatrix} \in \mathcal C\cdot \mathcal D 
 $$ 
and 
$$
\ell(p,2)^{\frac1p}=\begin{pmatrix} a&b\\b&p^2c\end{pmatrix}\simeq \begin{pmatrix} p^2c&-b\\-b&a\end{pmatrix} = \begin{pmatrix} pc&-b\\-b&pa\end{pmatrix} \cdot \begin{pmatrix}p&-b\\-b&pac\end{pmatrix}\in \mathcal C\cdot \mathcal D^{-1}.
 $$ 
This completes the proof. \end{proof}
%%%%%%%%%%%%%%%%%%%%%%%%%%%%%%%%%%%%%%%%%%%%%%%%%%%%%%%%%%%%%%%%%%%%%%%%%%%%%%%%%%%%%%%%
%%%%%%%%%%%%%%%%%%%%%%%%%%%%%%%%%%%%%%%%%%%%%%%%%%%%%%%%%%%%%%%%%%%%%%%%%%%%%%%%%%%%%%%%
%%%%%%%%%%%%%%%%%%%%%%%%%%%%%%%%%%%%%%%%%%%%%%%%%%%%%%%%%%%%%%%%%%%%%%%%%%%%%%%%%%%%%%%%
%%%%%%%%%%%%%%%%%%%%%%%%%%%%%%%%%%%%%%%%%%%%%%%%%%%%%%%%%%%%%%%%%%%%%%%%%%%%%%%%%%%%%%%%
%%%%%%%%%%%%%%%%%%%%%%%%%%%%%%%%%%%%%%%%%%%%%%%%%%%%%%%%%%%%%%%%%%%%%%%%%%%%%%%%%%%%%%%%

\begin{thm} \label{maint} Let $\ell$ be an irrecoverable binary $\z$-lattice with $\mathfrak{n}\ell=s\z$ and $D_\ell=s^2 D$, and let $L$ be an isolation of $\ell$. Then for any prime $p$ such that 
\[
	\left(\frac {D}{p}\right)=1 \text{ and } p\in Q(\mathcal D) \text{ for some ambiguous class } \mathcal D\in\mathfrak S_D,
\] 
$L$ has a binary primitive sublattice whose norm is in $p\z$. 
\end{thm}

\begin{proof}  
Let us assume that 
$$
\ell^{\frac1s}=\begin{pmatrix} pa&b\\ b&pc \end{pmatrix} \in \mathcal C 
$$
for some $\mathcal C \in \mathfrak S_D$. If we define $\ell^{\frac1s}(p)$ to be one of the sublattices of $\ell^{\frac1s}$ with index $p$ whose norm is contained in $p\z$, then by Lemma \ref{useful}, we have
$$
\left(\ell^{\frac1{s}}(p)\right)^{\frac{1}{p}} \in \mathcal C\cdot\mathcal D \simeq \mathcal C\cdot\mathcal D^{-1}.
$$

Now we claim that if there is a binary $\z$-lattice $\ell' \in \mathcal A$ with $\mca\in\mathfrak{S}_D$ such that $[\ell': \ell^{\frac1{s}}(p)]=p$, then $\ell^{\frac1s} \simeq \ell'$. The claim follows from Lemma \ref{useful} as
$$
\left(\ell^{\frac1{s}}(p)\right)^{\frac{1}{p}} \simeq \ell'(p,1)^{\frac1p} \in \mathcal A\cdot\mathcal D\quad \text{or} \quad \left(\ell^{\frac1{s}}(p)\right)^{\frac{1}{p}} \simeq \ell'(p,2)^{\frac1p} \in \mathcal A\cdot\mathcal D^{-1}.
$$
Thus we have $\mathcal A \simeq \mathcal C$, that is, $\ell^{\frac1s} \simeq \ell'$.

Let $\ell(p)$ be a binary $\z$-sublattice of $\ell$ with index $p$ whose norm is in $p\z$. Since $L$ is  an isolation of $\ell$, there is an isometry $\phi : \ell(p) \to L$. Let us define 
$$
\ell_1=\q (\phi(\ell(p))) \cap L \quad   \text{and}  \quad t=[\ell_1:\phi(\ell(p))].
$$ 

We now show that $t$ is not divisible by $p$ which implies that $\mathfrak {n}\ell_1 \subseteq p\z$, and hence implies the theorem. Suppose on the contrary that $t$ is divisible by $p$. Then there is a binary $\z$-sublattice $\ell_2$ of $L$ such that $[\ell_2:\phi(\ell(p))]=p$. Note that $\mathfrak{n}\ell_2=s\z$. Hence applying the claim, we have $\ell_2 \simeq \ell$, which contradicts to the assumption that $\ell$ is not represented by $L$.
\end{proof}
%%%%%%%%%%%%%%%%%%%%%%%%%%%%%%%%%%%%%%%%%%%%%%%%%%%%%%%%%%%%%%%%%%%%%%%%%%%%%%%%%%%%%%%%
%%%%%%%%%%%%%%%%%%%%%%%%%%%%%%%%%%%%%%%%%%%%%%%%%%%%%%%%%%%%%%%%%%%%%%%%%%%%%%%

\begin{cor} \label{ternary} Let $\ell$ be an irrecoverable binary $\z$-lattice. Then the rank of any isolation of $\ell$ is greater than $3$.
\end{cor}

\begin{proof} Assume that $\mathfrak{n}\ell=s\z$ and $D_\ell=s^2D$. Let $p$ be any prime which is represented by the identity class $\mathcal E \in \mathfrak S_D$. Then $\left(\frac {D_\ell}p\right)=1$.
Assume that there is a  ternary isolation, say $L$, of $\ell$. Then by Theorem \ref{maint}, there is a basis $\{\bm{x},\bm{y}, \bm{z}\}$ for $L$ such that 
$$
M_L \equiv \begin{pmatrix} 0&0&B(\bm{x},\bm{z})\\ 0&0&B(\bm{y},\bm{z})\\ B(\bm{x},\bm{z})&B(\bm{y},\bm{z})&Q(\bm{z}) \end{pmatrix} \Mod p,
$$
which implies that $dL \equiv 0 \Mod p$. Since there are infinitely many such primes $p$, we have a contradiction. 
\end{proof}
%%%%%%%%%%%%%%%%%%%%%%%%%%%%%%%%%%%%%%%%%%%%%%%%%%%%%%%%%%%%%%%%%%%%%%%%%%%%%%%%%%%%%%%%
%%%%%%%%%%%%%%%%%%%%%%%%%%%%%%%%%%%%%%%%%%%%%%%%%%%%%%%%%%%%%%%%%%%%%%%%%%%%%%%%%%%%%%%%
%%%%%%%%%%%%%%%%%%%%%%%%%%%%%%%%%%%%%%%%%%%%%%%%%%%%%%%%%%%%%%%%%%%%%%%%%%%%%%%%%%%%%%%%
%%%%%%%%%%%%%%%%%%%%%%%%%%%%%%%%%%%%%%%%%%%%%%%%%%%%%%%%%%%%%%%%%%%%%%%%%%%%%%%%%%%%%%%%
%%%%%%%%%%%%%%%%%%%%%%%%%%%%%%%%%%%%%%%%%%%%%%%%%%%%%%%%%%%%%%%%%%%%%%%%%%%%%%%%%%%%%%%%

 \begin{cor} \label{cor:quat} Let $\ell$ be an irrecoverable binary $\z$-lattice and let $\mathfrak{n}\ell=s\z$ and $D_\ell=s^2D$. If a quaternary $\z$-lattice $L$ is an isolation of $\ell$, then any prime $q$ with $\ord_q(dL) \equiv 1 \Mod 2$  is a divisor of $D$.   
   \end{cor}
   
 \begin{proof} 
 Suppose on the contrary that there is a prime $q$  such that  $\ord_q(dL) \equiv 1 \Mod 2$ and $(q,D)=1$. Then there are infinitely many primes $p$  such that
 $$
 \left(\frac{dL}p\right)=-1, \ \  \left(\frac{D}p\right)=1, \ \  \text{and} \ \ p \in Q(\mathcal E_D), 
 $$
 by Meyer's Theorem in \cite{m}. Then by Theorem \ref{maint}, there is a basis $\{\bm{x},\bm{y},\bm{z}, \bm{w}\}$ for $L$ such that $\mathfrak n(\z\bm{x}+\z\bm{y}) \subseteq p\z$.  
  Hence we have 
$$
M_L \equiv \begin{pmatrix} 0&0&B(\bm{x},\bm{z})&B(\bm{x},\bm{w})\\ 0&0&B(\bm{y},\bm{z})&B(\bm{y},\bm{w})\\ B(\bm{x},\bm{z})&B(\bm{y},\bm{z})&Q(\bm{z})&B(\bm{z},\bm{w})\\ B(\bm{x},\bm{w})&B(\bm{y},\bm{w})&B(\bm{z},\bm{w})&Q(\bm{w}) \end{pmatrix} \Mod p.
$$
Therefore we have
$$
dL \equiv (B(\bm{x},\bm{z})B(\bm{y},\bm{w})-B(\bm{x},\bm{w})B(\bm{y},\bm{z}))^2 \Mod p,
$$
which is a contradiction to the assumption.
\end{proof}

\begin{cor}\label{4-square}  Let $\ell$ be a binary $\z$-lattice such that $\mathfrak{n}\ell=s\z$ and $D_\ell=s^2D$.  If there does not exist a class of order $4$ in $\mathfrak S_D$, then the discriminant of any quaternary isolation, if exists, of $\ell$ is a square of an integer.  
\end{cor}

\begin{proof} Note that the group $\mathfrak S_D$ has no class of order $4$ if and only if any genus in $\mathfrak S_D$ contains an ambiguous class. Hence the corollary follows directly from  Lemma \ref{easy} and  Theorem \ref{maint}.  \end{proof}

\begin{rmk}
Despite the identification of numerous potential candidates for quaternary isolations of certain irrecoverable binary $\z$-lattices, no example has yet been proven to be an isolation. In Table \ref{tablecand4}, we provide all candidates for quaternary isolations of some binary $\z$-lattices. We have checked that each of the quaternary candidates represents all sublattices of $\ell$ with index $p$, where $p$ is a prime no greater than $149$. 

Note that the first three candidates for quateranry isolations of 
\[
    \ell=\begin{pmatrix} 6&3\\3&8 \end{pmatrix}
\] 
in Table \ref{tablecand4} have nonsquare discriminants. This is possible since $\ell$ does not satisfy the conditions in Corollary \ref{4-square}, as $\mathfrak{S}_D$ has a class of order $4$. 
\end{rmk}

\begin{table}[ht]
\caption{All candidates for quaternary isolations of some binary $\z$-lattices}\label{tablecand4}
\begin{tabular}{c|c|l}
\hline 
$\ell$ & $\mathfrak{S}_{D}$ & Candidates for quaternary isolations of $\ell$ with their discriminants\\
\hline \rule{0pt}{2.5em}
\multirow{8}{*}{$\begin{pmatrix}2&1\\1&2\end{pmatrix}$} & \multirow{8}{*}{$\{\mathcal{E}\}$}
&
\begin{scriptsize}
$\begin{pmatrix}1&0&0&0\\0&1&0&0\\0&0&4&2\\0&0&2&5\end{pmatrix}_{16}$  
$\begin{pmatrix}1&0&0&0\\0&2&0&1\\0&0&2&1\\0&1&1&5\end{pmatrix}_{16}$ 
$\begin{pmatrix}1&0&0&0\\0&2&1&-1\\0&1&5&1\\0&-1&1&5\end{pmatrix}_{36}$ 
$\begin{pmatrix}2&0&1&1\\0&3&0&0\\1&0&3&0\\1&0&0&3\end{pmatrix}_{36}$ \end{scriptsize}\\[2em]
&&
\begin{scriptsize}
$\begin{pmatrix}2&0&0&1\\0&2&0&1\\0&0&4&2\\1&1&2&6\end{pmatrix}_{64}$ 
$\begin{pmatrix}2&0&1&0\\0&3&1&1\\1&1&5&2\\0&1&2&5\end{pmatrix}_{100}$ 
$\begin{pmatrix}2&1&0&1\\1&3&1&1\\0&1&5&-2\\1&1&-2&6\end{pmatrix}_{100}$ 
$\begin{pmatrix}2&0&1&1\\0&4&2&-2\\1&2&6&1\\1&-2&1&6\end{pmatrix}_{144}$ 
\end{scriptsize}\\[2em]
&&
\begin{scriptsize}
$\begin{pmatrix}2&0&1&1\\0&6&1&-1\\1&1&6&3\\1&-1&3&6\end{pmatrix}_{256}$ 
$\begin{pmatrix}2&0&1&1\\0&6&3&-1\\1&3&6&2\\1&-1&2&10\end{pmatrix}_{400}$
\end{scriptsize}\\[1.8em]
\hline \rule{0pt}{2.5em}
$\langle 2,2\rangle$ & $\{\mathcal{E}\}$ &
\begin{scriptsize}
$\begin{pmatrix}1&0&0&0\\0&2&0&1\\0&0&4&0\\0&1&0&5\end{pmatrix}_{36}$ 
$\begin{pmatrix}1&0&0&0\\0&2&0&1\\0&0&4&2\\0&1&2&6\end{pmatrix}_{36}$ 
$\begin{pmatrix}2&1&1&1\\1&4&0&1\\1&0&4&0\\1&1&0&4\end{pmatrix}_{81}$ 
$\begin{pmatrix}2&0&0&1\\0&4&0&2\\0&0&4&0\\1&2&0&6\end{pmatrix}_{144}$
\end{scriptsize}\\[1.8em]
\hline \rule{0pt}{2.5em}
$\begin{pmatrix}2&1\\1&3\end{pmatrix}$ & $\z / 3\z$ &
\begin{scriptsize}
$\begin{pmatrix}1&0&0&0\\0&2&0&1\\0&0&6&1\\0&1&1&6\end{pmatrix}_{64}$
\end{scriptsize}\\[1.8em]
\hline \rule{0pt}{2.5em}
$\begin{pmatrix}6&3\\3&8\end{pmatrix}$ & $\z /4\z$ &
\begin{scriptsize}
$\begin{pmatrix}1&0&0&0\\0&2&1&1\\0&1&3&0\\0&1&0&11\end{pmatrix}_{52}$ 
$\begin{pmatrix}2&0&1&1\\0&3&1&1\\1&1&4&2\\1&1&2&4\end{pmatrix}_{52}$ 
$\begin{pmatrix}2&0&1&1\\0&4&2&-2\\1&2&4&0\\1&-2&0&12\end{pmatrix}_{208}$ 
$\begin{pmatrix}6&2&3&3\\2&6&0&1\\3&0&10&5\\3&1&5&12\end{pmatrix}_{2401}$
\end{scriptsize}\\[1.8em]
\hline
\end{tabular}
\end{table}

%%%%%%%%%%%%%%%

\end{document}